\documentclass{article}
\usepackage{amsthm, amsfonts, amssymb, latexsym, amsmath}
\usepackage{xcolor}
\usepackage{ulem}

\newcommand{\F}{\mathbb{F}}
\DeclareMathOperator{\Tr}{Tr}

\DeclareMathOperator{\modulo}{mod}
\renewcommand{\mod}{\, \modulo \,}

\newtheorem{thm}{Theorem}
\newtheorem{lemma}[thm]{Lemma}

\theoremstyle{definition}

\newtheorem{rem}[thm]{Remark}

\newcommand{\red}[1]{{{\color{black}#1}}}
\newcommand{\brown}[1]{{{\color{black}#1}}}

\begin{document}
\title{Additive double character sums over some structured sets and applications}

\author{Cathy Swaenepoel$^{1,2}$ and Arne Winterhof$^3$}
\date{}

\maketitle

 
 \begin{center}
$^1$ Aix Marseille Universit\'e, CNRS, Centrale Marseille, I2M, Marseille, France\\
$^2$ Institut de Math\'ematiques de Marseille UMR 7373 CNRS, 163 Avenue de Luminy, Case 907,
13288 Marseille Cedex 9, France\\
E-mail: cathy.swaenepoel@univ-amu.fr\\
$^3$ Johann Radon Institute for Computational and Applied Mathematics,
Austrian Academy of Sciences, Altenberger Str. 69, 4040 Linz, Austria\\
E-mail: arne.winterhof@oeaw.ac.at
\end{center}
 
 \begin{abstract}
We study additive double character sums over two subsets of a finite field. We show that if there is a suitable rational self-map of small degree of a set $D$, then this set contains a large subset $U$ for which the standard bound on the absolute value of the character sum over $U$ and any subset~$C$ (which satisfies some restrictions on its size $|C|$) can be improved.
  Examples of such suitable self-maps are inversion and squaring.
  Then we apply this new bound to trace products and sum-product equations and improve results of the first author and of Gyarmati and S\'ark\"ozy.
 \end{abstract}



 \section{Introduction}
 
  \brown{We denote by $\F_q$ the finite field with $q$ elements and by $p$ the characteristic of~$\F_q$. We write $\Tr$ for the absolute trace of~$\F_q$ and we use the notation
  $Y\ll Z$ and $Z\gg Y$ if $|Y|\le \lambda Z$ for some constant $\lambda >0$.}
  
  For a non-trivial additive character \brown{$\psi$} of $\F_q$ and any subsets
  $C,D\subseteq\F_q$ it is well-known that
  \begin{equation}\label{general}
      \left|\sum_{c\in C,d\in D} \psi(cd)\right|\le \min\left\{\left(|C||D|q\right)^{1/2},|C||D|\right\},
  \end{equation} 
  see for example \cite[Corollary 1]{gysa08}.
  
  This bound is tight in general, for example, if $q$ is a square, take $C=D=\F_{q^{1/2}}$ and $\psi$ be any non-trivial additive character of $\F_q$ which is trivial on the subfield $\F_{q^{1/2}}$. 
  
 We also provide an example when $q=p$ is a prime. Take 
 $$C=D=\left\{0,1,2,\ldots,\left\lfloor 0.1 p^{1/2}\right\rfloor\right\}.$$
 Then we have $0\le cd\le 0.01 p$ for any $(c,d)\in C\times D$ and thus for the additive canonical character
 $$\psi(x)=\exp\left(\frac{2 \pi i x}{p}\right)=\cos\left(\frac{2\pi x}{p}\right)+i\sin\left(\frac{2\pi x}{p}\right)$$
 of $\F_p$ we get
 $$\left|\sum_{c\in C,d\in D}\psi(cd)\right|\ge |C||D|\cos(0.02\pi)\ge 0.99 |C||D|$$
 which is of the same order of magnitude $p$ as $\sqrt{|C||D|p}$.
  
  Note that the bound $(\ref{general})$ is only non-trivial if $|C||D|>q$. However, if $C$ or $D$ is a structured set such as a multiplicative subgroup of $\F_q^*$ or an additive subgroup of~$\F_q$ we know better bounds, see for instance \cite{boglko,oswi,wi} and references therein.
  For example, if $q=p$ is a prime and $D$ is any subgroup of $\F_p^*$ of order of magnitude at least $p^\varepsilon$, the bound of \cite[Theorem 6]{boglko} on $\left|\sum\limits_{d\in D}\psi(cd)\right|$ is nontrivial for any $c\in \F_p^*$ which immediately gives a nontrivial bound on the absolute value of the double sum 
  $\sum\limits_{c\in C,d\in D}\psi(cd)$ for very small multiplicative subgroups $D$ of $\F_p^*$ and arbitrary
  $C\subseteq \F_p^*$.
  


  In this paper, 
   we show that if $D$ has some desirable structure, then there is a large subset $U$ of $D$ for which we obtain a better upper bound on $$\left|\sum_{c\in C,u\in U} \psi(cu)\right|$$ 
   than the classical bound \eqref{general}.
  The needed structure is the existence of a rational function $f(X)\in \F_q(X)$ with $f(D)\subseteq D$ of small degree which satisfies a certain property of nonlinearity: 
  \begin{equation}\label{cond_f}
      f(X) \notin \{a(g(X)^p-g(X))+bX+c: g(X)\in \F_q(X), a,b,c \in \F_q\}\brown{.}
  \end{equation}
  If $f\in \F_q(X)\setminus\{0\}$, then the degree of $f$ is defined as $\max\{\deg g, \deg h\}$ where $g$ and $h$ are coprime polynomials on $\F_q$ such that $h\neq 0$ and $f=g/h$.
  Examples for the choice of $f(X)$ are $f(X)=X^{-1}$ and $f(X)=X^2$ for odd $q$. 
  
  More precisely, we prove the following bound on double character sums. 
  \begin{thm}\label{charsumbound}
  Let $D\subseteq \F_q$ with $|D| \geq 2$ and assume that there exists $f(X) \in \F_q(X)$ of degree $k$ 
  satisfying \eqref{cond_f}
  such that 
  $f(X)$ has no pole in $D$ and $f(D)\subseteq D$.
Then there exists $U \subseteq D$ such that 
$$|U|\geq \frac{|D|}{k+1}$$ 
and 
    for any $C\subseteq \F_q$ and any non-trivial additive character $\psi$ of $\F_q$
we have
\begin{equation}\label{mainbound}
\left|\sum_{c\in C, u\in U} \psi(c u)\right| 
\ll \left( \frac{|C|^3 |D|^3 q}{M(|D|)} \right)^{1/4},
\end{equation}
where the implied constant depends only on $k$ and we used the abbreviation
\begin{equation}\label{M} M(|D|)=\min\left\{ \frac{q^{1/2}}{|D|^{1/2} (\log|D|)^{11/4}},\frac{|D|^{4/5}}{q^{2/5}(\log|D|)^{31/10}} \right\}.
\end{equation}
  \end{thm}

Note that there exists a constant $\lambda >0$ (depending only on $k$) such that $(\ref{mainbound})$ improves the classical bound $(\ref{general})$
if $|C|>0$ and 
$$M(|D|)> \lambda \max\left\{\frac{q}{|C||D|},\frac{|C||D|}{q}\right\},$$
which is satisfied if 
\begin{eqnarray*}&&\lambda \max\left\{ \frac{q^{1/2}(\log q)^{11/4}}{|D|^{1/2}}, \frac{q^{7/5}(\log q)^{31/10}}{|D|^{9/5}}\right\} < |C| \\&& <\lambda^{-1} \min\left\{ \frac{q^{3/2}}{|D|^{3/2}(\log q)^{11/4}},\frac{q^{3/5}}{|D|^{1/5}(\log q)^{31/10}} \right\}
\end{eqnarray*}
which is a nonempty interval for $|C|$ if 
$$(2\lambda^2)^{5/8} q^{1/2}(\log q)^{31/8} < |D| < (2\lambda^{2})^{-1}q(\log q)^{-11/2}$$
and $q$ is sufficiently large.



\begin{rem}
\label{rem1}

\brown{
The condition \eqref{cond_f} on $f(X)$ cannot be removed in Theorem \ref{charsumbound}.
Indeed, without this condition, we could take $f(X)=X$ and choose
\begin{equation*}
    C=\left\{0,1,2,\ldots,\left\lfloor 0.1 p^{1/2}\right\rfloor\right\},
    \quad 
    \quad
    D=\{x\in \F_q:\Tr(x)\in C 
    \}
\end{equation*}
and the additive character $\psi$ of $\F_q$ defined by
$$\psi(x)=\exp\left(\frac{2\pi i \Tr(x)}{p}\right).$$
Since $0 \leq c\Tr(d) \leq 0.01p$ for any $(c,d)\in C\times D$, we have for any $U \subseteq D$,
$$\left|\sum_{c\in C, u\in U} \psi(cu)\right|
\geq |C||U| \cos(0.02 \pi) \ge 0.99 |C||U|.
$$
Moreover, since $0.1 q/p^{1/2}\leq |D|=|C|q/p \leq q/p^{1/2}$, 
there exist absolute constants $\lambda_1,\lambda_2>0$ such that if 
$$\lambda_1 (\log q)^{11} < p < \lambda_2 q(\log q)^{-31/4},$$ then the upper bound in \eqref{mainbound} is smaller than $0.99|C||D|/2$,
which implies that~\eqref{mainbound} holds for no $U$ with $|U|\ge |D|/2$.}

\end{rem}

We prove Theorem~\ref{charsumbound} in Section~\ref{proof}.
Furthermore, we apply Theorem~\ref{charsumbound} to two problems:

  \begin{enumerate}
     \item For $C,D\subseteq \F_q$ find conditions such that
     \begin{equation}\label{trcd} \Tr(CD)=\{\Tr(cd) :c\in C,d\in D\}=\F_p\brown{.}
     \end{equation}
     The first author proved in \cite{sw18} that if $C,D\subseteq \F_q^*$, then
     \begin{equation}\label{cond_ini_trace}
         |C||D|>p^2q
     \end{equation}
     implies $(\ref{trcd})$.
     The condition \eqref{cond_ini_trace} is in general optimal up to an absolute constant factor (see \cite[Section 3.6]{sw18}). 
     
     In Section~\ref{trsec} we use Theorem~\ref{charsumbound} to relax this condition for many sets $C$ and~$D$.
     
    \item Gyarmati and S\'ark\"ozy used $(\ref{general})$ in \cite{gysa2} to show that for any 
    sets $A,B,C,D\subseteq \F_q$ with 
    \begin{equation}\label{q3}|A||B||C||D|>q^3
    \end{equation}
    there is a solution $(a,b,c,d)\in A\times B\times C\times D$
    of the sum-product equation
    $$a+b=cd.$$
    The condition \eqref{q3} is in general optimal up to an absolute constant factor.
    For instance if the characteristic is at least 3 and $q\equiv 1 \mod 4$, 
    \red{let $\{x_1,\ldots,x_{(q-1)/\brown{4}}\}\subseteq \F_q^*$ be any set with \brown{$(q-1)/4$ elements satisfying } $x_i\ne -x_j$, $1\le i< j\le (q-1)/\brown{4}$,} and take $A=\bigcup\limits_{i=1}^{(q-1)/4}\{x_i,-x_i\}$, $B=\F_q^*\setminus A$, $C=\F_q$ and $D=\{0\}$.
    
    In Section~\ref{sumprod} we relax $(\ref{q3})$ under the conditions of Theorem~\ref{charsumbound}.
\end{enumerate}

 \section{Proof of Theorem~\ref{charsumbound}}
 \label{proof}
 
 The proof of Theorem~\ref{charsumbound} is a combination of a bound on character sums in terms of additive energy and an existence result of a large subset of small additive energy.

 For $S\subseteq \F_q$, let $E(S)$ denote its additive energy, that is, 
 the number of solutions $(s_1,s_2,s_3,s_4)\in S^4$ of $s_1+s_2=s_3+s_4$.

Our first tool is an improvement of the bound $(\ref{general})$ on double additive character sums provided
that one of the summation sets is of small additive energy. 
\begin{lemma}\label{lem_upbound_character_sum}
Let $C$ and $U$ be subsets of $\F_q$ and $\psi$ be a non-trivial additive character of $\F_q$. Then
we have
$$
\left|\sum_{c\in C, u\in U} \psi(c u)\right| 
\leq \left( |C|^3 E(U) q \right)^{1/4}.
$$
\end{lemma}

\begin{proof}
By H\"older's inequality we have
 \begin{eqnarray*} && \left|\sum_{c\in C,u\in U}\psi(c u)\right|^4\\
 &\le&\left(\sum_{c\in C}\left|\sum_{u\in U}\psi(c u)\right|\right)^4
 \le |C|^3\sum_{c\in \F_q} \left|\sum_{u\in U}\psi(c u)\right|^4\\
 &=&|C|^3\sum_{u_1,u_2,u_3,u_4\in U} \sum_{c\in \F_q}\psi(c(u_1+u_2-u_3-u_4))\\
 &=&|C|^3 E(U)q
 \end{eqnarray*}
 and the result follows.
\end{proof}
 
 Our second tool is the following consequence of a decomposition theorem of Roche-Newton, Shparlinski and the second author \cite[Theorem 1.1]{roshwi}. 
\begin{lemma}\label{lem_existence_U}
Let $D\subseteq \F_q$ with $|D| \geq 2$ and assume that there exists a rational function $f(X) \in \F_q(X)$ of degree $k$
satisfying \eqref{cond_f}
such that $f(X)$ has no pole in $D$ and $f(D)\subseteq D$.
Then there exists $U \subseteq D$ such that 
$$|U|\geq \frac{|D|}{k+1}$$ 
and 
\begin{equation}\label{upbound_energy_U}
E(U) \ll \frac{|D|^3}{M(|D|)},
\end{equation}
where the implied constant depends only on $k$ and $M(|D|)$ is defined by $(\ref{M})$.
\end{lemma}

\begin{proof}
By \cite[Theorem 1.1]{roshwi}, there exist disjoint sets $S$ and $T$ with $D=S\cup T$
and
\begin{equation}\label{energy} 
\max\{E(S),E(f(T))\}\ll \frac{|D|^3}{M(|D|)},
\end{equation}
where the implied constant depends only on $k$.
If $|S|\ge |D|/(k+1)$ then we take $U=S$. Otherwise, we have $|T|\ge |D|k/(k+1)$ and $|f(T)|\ge |D|/(k+1)$ since each value of $f(X)$ can be attained at most $k$ times. In this case, we take $U=f(T)$. 
In both cases, we have $U\subseteq D$ and by \eqref{energy}, $U$ satisfies \eqref{upbound_energy_U}.
\end{proof}


Now Theorem~\ref{charsumbound} follows by combining Lemmas \ref{lem_upbound_character_sum} and \ref{lem_existence_U}.



 \section{Trace products}
 \label{trsec}
 

\begin{thm}\label{thm_trace}
Let $C,D\subseteq \F_q$ with $|D| \geq 2$ and assume that there exists a rational function $f(X) \in \F_q(X)$ of degree $k$ satisfying \eqref{cond_f} such that $f(X)$ has no pole in $D$ and $f(D)\subseteq D$. There exists a constant $\lambda>0$ depending only on $k$ such that if   
\begin{equation}\label{lowerbound_trace}
|C|> \lambda \frac{p^4q}{|D|M(|D|)},
\end{equation}
then 
$$\Tr(CD)=\F_p,$$
where $M(|D|)$ is defined by $(\ref{M})$.
\end{thm}

\begin{proof}
Let $U$ be as in Theorem~\ref{charsumbound} and for $s \in \F_p$, let $N_s$ be the number of solutions $(c,u)\in C\times U$ of $\Tr(cu)=s$. 
It follows from \eqref{mainbound} that there exists a constant $\kappa>0$ depending only on $k$ such that for any $s\in\F_p$, we have
\begin{align*}
\left| N_s- \frac{|C||U|}{p}\right|
&= \left|
\frac{1}{p} \sum_{j=1}^{p-1} \exp\left( \frac{-2i\pi js}{p} \right) \sum_{c\in C, u \in U} \exp\left( \frac{2i\pi j\Tr(cu)}{p} \right)\right|\\
&\quad 
\leq \kappa \left(  \frac{ |C|^3 |D|^3 q}{M(|D|)}  \right)^{1/4}.
\end{align*}
Since $$\frac{|C||U|}{p} \geq \frac{|C||D|}{(k+1)p},$$ 
the condition 
$$  |C| >  \kappa^4 (k+1)^4  \frac{p^4 q}{|D| M(|D|)} $$
implies 
$N_s>0$ for any $s\in \F_p$, that is, $\Tr(CU)=\F_p$ and thus $\Tr(CD)=\F_p$.
Taking $\lambda = \kappa^4 (k+1)^4$, 
the result follows.
\end{proof}

\begin{rem}
For $C,D\subseteq \F_q^*$, \cite[Theorem 1.1]{sw18} implies $\Tr(CD)=\F_p$ if $(\ref{cond_ini_trace})$ is satisfied
and by \cite[Theorem 1.2]{sw18} we have
\begin{equation}\label{tr2} \F_p^*\subseteq \Tr(CD)\quad \mbox{if}\quad |C||D|\ge pq.
\end{equation}

If 
\begin{equation*}
\lambda^{5/4}  p^{5/2}q^{1/2}(\log q)^{31/8} < |D| <  \frac{q}{\lambda^2 p^4(\log q)^{11/2}}
\end{equation*}
(with $\lambda$ as in Theorem \ref{thm_trace}), then \eqref{lowerbound_trace} defines a larger range of $|C|$ with $\Tr(CD)=\F_p$ than $\eqref{cond_ini_trace}$.
\brown{In the case where the degree of $f$ is an absolute constant (so $\lambda$ is also an absolute constant), n}ote that this range for~$|D|$ is non-trivial if $q=p^r$ with $r\ge 14$ and $q$ is sufficiently large.


Similarly we get an improvement of $(\ref{tr2})$  if 
$$\lambda^{5/4} p^{15/4}q^{1/2}(\log q)^{31/8} < |D| < \frac{q}{\lambda^2 p^6(\log q)^{11/2}},$$
which, \brown{in the case where the degree of $f$ is an absolute constant}, is a nontrivial range for $|D|$ if $q=p^r$ with $r\ge 20$ and $q$ is sufficiently large.
\end{rem}

\begin{rem}
For any $s \in \F_p$, the number of $x\in \F_q$ such that $\Tr(x)=s$ is $q/p$. It follows that if $A$ is a subset of $\F_q$ with $|A|\leq q/p$ then $\Tr(A)$ may contain only 
one element of $\F_p$.

Theorem \ref{thm_trace} ensures that $\Tr(CD)$ covers $\F_p$ for many sets $C$ and $D$ such that~$|CD|$ 
is much smaller than $q/p$.
For instance, if $r\geq 20$ and $q$ is sufficiently large, then for any $D$ closed under inversion with 
$$|D|=\left\lceil q^{9/13} (\log q)^{7/26} \right\rceil$$ 
and any $C$ with 
$$|C|= \left\lceil p^4 q^{2/13}(\log q)^{69/26}\right\rceil$$
we have $\Tr(CD)=\F_p$ and 
$$|C||D| \ll q^{\alpha}(\log q)^{38/13}/p, \quad \mbox{where }\alpha = 3/r+11/13 <1.$$
\end{rem}

\begin{rem}
\brown{The following example shows that the condition \eqref{cond_f} on $f(X)$ cannot be removed in Theorem \ref{thm_trace}.}
Assume that $q=p^r$ with $p$ odd. Let $\{1,x,\ldots,x^{r-1}\}$ be a power basis of $\F_q$ over $\F_p$, let $\{w_0,\ldots,w_{r-1}\}$ be its dual basis, that is,
$$\Tr(w_ix^j)=\left\{\begin{array}{cc} 1, & i=j,\\ 0, & i\ne j,\end{array}\right. \quad 0\le i,j<r,$$
and denote by $\mathcal{Q}=\{z^2:z\in \F_p^*\}$ the set of
quadratic residues modulo $p$.
Assume that $r\equiv 0\mod 4$, take $f(X)=X$ and
 consider
 \begin{equation*}
    C = w_0\left(\mathcal{Q}+\F_p x + \cdots + \F_p x^{r/4}\right),
    \quad
    D = \mathcal{Q}+\F_p x + \cdots + \F_p x^{(3/4)r-1}.
\end{equation*}
Then, $|C|=\frac{p-1}{2} p^{r/4}$ and $|D|=\frac{p-1}{2} p^{(3/4)r-1}$. Hence
$$
\frac{|C||D|M(|D|)}{p^4q} \gg \frac{q^{1/8}}{p^3(\log q)^{31/10}}.
$$
\brown{It follows that for any constant $\lambda>0$, if $r\geq 28$ and $q$ is sufficiently large, then $C$ and $D$ satisfy \eqref{lowerbound_trace}.}
Moreover, 
\begin{eqnarray*}
\Tr(CD)&\subseteq& \Tr(w_0(\mathcal{Q}+\F_p x + \cdots + \F_p x^{r-1}))\\
&=&\Tr(w_0)\mathcal{Q}+\Tr(w_0x)\F_p+\cdots+\Tr(w_0x^{r-1})\F_p
= \mathcal{Q}
\varsubsetneq \F_p.
\end{eqnarray*}
\end{rem}

\section{The equation $a+b=cd$}\label{sumprod}

Now we prove an improvement of \cite[Corollary 1]{gysa2} (see also \cite{sa} for prime fields).

\begin{thm}\label{thm_eq}
Let $A,B,C,D\subseteq \F_q$ and denote by $N$ the number of solutions $(a,b,c,d)\in A\times B \times C \times D$ of the equation
\begin{equation*}
    a+b=cd.
\end{equation*}
Assume that $|D| \geq 2$ and that there exists a rational function $f(X) \in \F_q(X)$ of degree $k$ 
satisfying \eqref{cond_f}
such that $f(X)$ has no pole in $D$ and $f(D)\subseteq D$. Then there exists a constant $\lambda>0$ depending only on $k$ such that 
if
\begin{equation}\label{lowerbound_eq}
|A|^2|B|^2|C||D|M(|D|)
>\lambda q^5,
\end{equation}
then $N>0$,
where $M(|D|)$ is defined by $(\ref{M})$.
\end{thm}

\begin{proof}

Let $U$ be as in Theorem~\ref{charsumbound} and let $N^*$ be the number of solutions $(a,b,c,u)\in A\times B \times C \times U$ of the equation $a+b=cu$.
It follows from \eqref{mainbound} and the Cauchy-Schwarz inequality that there is a constant $\kappa>0$ depending only on $k$ such that
\begin{eqnarray*}
\left|N^*-\frac{|A||B||C||U|}{q}\right|
&\leq&
\frac{1}{q}\sum_{\psi\neq \psi_0} \left|\sum_{a\in A} \psi(a)\right|\left|\sum_{b\in B} \psi(b)\right| \left|\sum_{c\in C,u\in U} \psi(cu)\right|
\\
&&\le
\frac{\kappa}{q} \left(\frac{|C|^3|D|^3q}{M(|D|)}\right)^{1/4}
\sum_{\psi\neq \psi_0} \left|\sum_{a\in A} \psi(a)\right|\left|\sum_{b\in B} \psi(b)\right|\\
&&\leq 
\frac{\kappa}{q} \left(\frac{|C|^3|D|^3q}{M(|D|)}\right)^{1/4} (q|A|)^{1/2} (q|B|)^{1/2}\\
&& \le \kappa\left(\frac{|A|^2|B|^2|C|^3|D|^3 q}{M(|D|)}\right)^{1/4}.
\end{eqnarray*}
Since $|U|\ge |D|/(k+1)$, 
$$|A|^2|B|^2|C||D|M(|D|) 
 >\kappa^4(k+1)^4 q^5$$ 
implies $N^*>0$ and thus $N>0$. Choosing $\lambda=\kappa^4(k+1)^4$ completes the proof.
\end{proof}

\begin{rem}
By \cite[Corollary 1]{gysa2},
we have $N>0$ if $(\ref{q3})$ is satisfied.

If 
\begin{equation*}
2\lambda\max\left\{  |D|^{1/2} q^{3/2} (\log q)^{11/4},\frac{q^{12/5}(\log q)^{31/10}}{|D|^{4/5}}\right\}
< |A||B| \leq q^2
\end{equation*}
(with $\lambda$ as in Theorem \ref{thm_eq}), which is a non-trivial range for $|A||B|$ if 
$$(4\lambda)^{5/4} q^{1/2} (\log q)^{31/8} < |D| 
< (4\lambda)^{-2}q(\log q)^{-11/2}$$ 
and $q$ is sufficiently large, then 
\eqref{lowerbound_eq} improves \eqref{q3}.
\end{rem}

\begin{rem}
\brown{The following two examples show that the condition \eqref{cond_f} on $f(X)$ cannot be removed in Theorem \ref{thm_eq}.}

\brown{For our first example, we a}ssume that $q=p^r$ with $p\equiv 1 \mod 4$ and $r\equiv 0\mod 4$. 
Let $\{1,x,\ldots,x^{r-1}\}$ be a power basis of $\F_q$ over $\F_p$.  Denote $\mathcal{Z}_1=\bigcup\limits_{i=1}^{(p-1)/4}\{\red{i,-i}\}$ and  $\mathcal{Z}_2=\F_p^*\setminus \mathcal{Z}_1$. 
Now take $f(X)=X$ and consider
\begin{align*}
     A &= \mathcal{Z}_1 + \F_p x + \ldots + \F_p x^{r-1}, 
    &B &= \mathcal{Z}_2 + \F_p x + \ldots + \F_p x^{r-1},\\
     C &= \F_p x + \cdots + \F_p x^{r/4},
    &D &= \F_p + \F_p x + \cdots + \F_p x^{(3/4)r-1}.
\end{align*}
Since $|A| = |B| \gg p^r$, $|C|= p^{r/4}$ and $|D|= p^{(3/4)r}$, we have
$$
\frac{|A|^2|B|^2|C||D|M(|D|)}{q^5} \gg \frac{q^{1/8}}{(\log q)^{31/10}}
$$
and it follows that for any constant $\lambda>0$, 
\brown{if $q$ is sufficiently large, then $|A|,|B|,|C|$ and $|D|$ satisfy \eqref{lowerbound_eq}.}
Moreover, 
$$A+B \subseteq \F_p^* + \F_p x + \ldots + \F_p x^{r-1}
\quad \mbox{and} \quad
CD \subseteq \F_p x + \ldots + \F_p x^{r-1}$$
which implies $N=0$.

\brown{For our second example, we take $f(X)=X$ and we use $C$ and $D$ from Remark~\ref{rem1}.}
Then 
$$\Tr(CD)\subseteq [0,0.01p]$$
and we take 
$$A=B=\{x\in \F_q : \Tr(x)\in (0.005p,p/2)\}.$$
Since 
$$\Tr(A+B)\cap \Tr(CD)=\emptyset$$
there is no solution of $a+b=cd$ with $(a,b,c,d)\in A\times B\times C\times D$.
However, $|A|^2|B|^2|C||D|$ is of order of magnitude $q^5$ and for any $\lambda>0$, there exist $\lambda_1,\lambda_2>0$ such that if 
$$\lambda_1 (\log q)^{11} < p < \lambda_2 q(\log q)^{-31/4},$$ then $\eqref{lowerbound_eq}$ is satisfied.
\end{rem}


 \section*{Acknowledgment\red{s}}
 The second author was partially funded by the Austrian Science Fund (FWF) project P 30405-N32.\\
\red{We wish to thank the anonymous referee for very useful comments.}

\end{document}